\documentclass{amsart}
\newtheorem{theorem}{Theorem}[section]

\usepackage{latexsym}
\usepackage{amssymb,amsmath,amsfonts}
\usepackage[colorlinks=true]{hyperref}
\usepackage{graphicx}
\usepackage[numbers,sort&compress]{natbib}
\usepackage{color}

\newtheorem{lemma}[theorem]{Lemma}

\newtheorem{assup}[theorem]{Assumptions}

\theoremstyle{definition}
\newtheorem{definition}[theorem]{Definition}

\newtheorem{example}[theorem]{Example}

\theoremstyle{remark}
\newtheorem{remark}[theorem]{Remark}

\numberwithin{equation}{section}
\global\parskip 6pt \oddsidemargin
0.1cm \evensidemargin 0.4cm \headheight 0.1cm \topmargin -2.0cm
\title[Dynamical System for Solving Mixed Variational Inequality problems.]
{Forward-Backward-forward Dynamical System for Solving Mixed Variational Inequality problems.}
\author{ Chidi Elijah Nwakpa$^{1},$  Chinedu Izuchukwu$^{2},$   Chibueze CHristian Okeke$^{3}.$ }
\address{$^{1,2,3}$School of Mathematics, University of the Witwatersrand, Private Bag 3, Johannesburg 2050, South Africa.}
\email{$^{1},$2713235@students.wits.ac.za, chidinecsc2n@gmail.com}
\email{$^{2},$chinedu.izuchukwu@wits.ac.za, $\mbox{izuchukwu}\_c$@yahoo.com}
\email{$^{3},$chibueze.okeke87@yahoo.com, chibueze.okeke@wits.ac.za}

\begin{document}
	\keywords{Mixed variational inequalities; dynamical system; Lyapunov analysis; generalized monotonicity;  pseudomonotone; global exponential stability.\\
		{\rm 2000} {\it Mathematics Subject Classification}: 47H09; 47H10; 49J53; 90C25}
	\begin{abstract}\noindent\\
		We study in this paper a forward-backward-forward dynamical system for solving mixed variational inequality problem in a real Hilbert space. For the convergence analysis of our proposed system, we apply the Lyapunov analysis  to obtain the weak convergence of the generated trajectories when the associated operator is Lipschitz continuous and satisfies the general monotonicity condition. We also assume that the involved real-valued convex function satisfies some mild assumptions. Furthermore, the Lipschitz continuous operator is taken to be $h-$strongly pseudomonotone to establish the global exponential stability of the equilibrium point of the system for all the orbits generated. Finally, we present some numerical examples which illustrate how the  trajectories of the proposed system converge to the equilibrium point of the proposed dynamical system.
		
		
	\end{abstract}
	
	\maketitle
	\section{Introduction}
	\noindent The mixed variational inequality theory, in its diverse forms, serves as a mathematical model for several interesting phenomena occurring in various research fields. The extension and generalization of the mixed variational inequality (briefly, MVI) to study a class of linear and nonlinear problems, which arise in various fields such as engineering, fluid flow through porous media, transportation models, image restoration, electronics, economics, financial analysis, physics, optimization, and others, have made the MVI a unifying mathematical formulation for finding approximate solutions to a class of linear and nonlinear operators
	 \cite{Facchinei1,Facchinei2,Goeleven1,Goeleven2,Han,Konnov}.
	
	\noindent Let $\mathcal{H}$ be a real Hilbert space endowed with the inner product function $\langle\cdot , \cdot\rangle$ and the corresponding norm operator $\| \cdot\|,$ and $\mathcal{C}\subset\mathcal{H}$ a nonempty closed convex set in $\mathcal{H}.$ Let $\mathcal{T}:\mathcal{H}\longrightarrow\mathcal{H}$ be a Lipschitz continuous operator and given $h:\mathcal{H}\longrightarrow\overline{\mathbb{R}} = [-\infty, +\infty],$ a proper, lower semicontinuous real-valued function, we consider the problem of finding $\bar{x}\in\mathcal{C}$ such that 
	\begin{eqnarray}\label{MVI1}
		\langle \mathcal{T}\bar{x}, \underline{u}-\bar{x}\rangle+h(\underline{u})-h(\bar{x})\geq0 \ \ \forall \underline{u}\in\mathcal{C}.
	\end{eqnarray}
	The problem described by \eqref{MVI1} is called the mixed variational inequality problem or variational inequality of the second kind, and in this work, we are going to approach the MVI\eqref{MVI1} from a continuous perspective.\\
	Let us denote $S^{MVI(\mathcal{T}; h; \mathcal{C})}:=\{\bar{x}\in\mathcal{C} ~:~ \langle \mathcal{T}(\bar{x}), \underline{u}-\bar{x}\rangle +h(\underline{u})-h(\bar{x})\geq 0 \ \ \forall \underline{u}\in \mathcal{C}\}$ to be the set of solutions of the MIV\eqref{MVI1}. 
	  Note that the function $h$ in \eqref{MVI1} may not necessarily be differentiable. In particular, $h$ may correspond to the indicator function $\delta_\mathcal{C}$ of a nonempty, closed convex subset $\mathcal{C}$ of $\mathcal{H}$, which may fail to possess the necessary regularity conditions for differentiability to hold. That is, $h= \delta_\mathcal{C},$ where
	  \begin{eqnarray*}
	  	\delta_\mathcal{C}(\underline{u}) = \begin{cases}
	  		0, \ \ \mbox{if} \ \underline{u}\in\mathcal{C},\\
	  		\infty, \ \ \mbox{otherwise.} 
	  	\end{cases}
	  \end{eqnarray*}
	  In this scenario, the MVI\eqref{MVI1} becomes the classical variational inequality problem (VIP) of Stampacchia \cite{STAM}, which is:\\
	  find $\bar{x}\in\mathcal{C}$ such that 
	  \begin{eqnarray}\label{1.2}
	  	\langle \mathcal{T}(\bar{x}), \underline{u}-\bar{x}\rangle\geq 0 \ \ \mbox{for each} \ \underline{u}\in\mathcal{C}.
	  \end{eqnarray}
	  Similarly, MVI\eqref{MVI1} reduces to the constrained optimization problem of minimizing $h$ over $\mathcal{C}$ if $\mathcal{T}\equiv 0$ (see, \cite{Jahn,Ricceri}).
	
	\noindent Furthermore, since MVI generalizes other models, a good number of methods have been proposed by several scholars as a means of finding an approximate solution to the problem \eqref{MVI1}. Over the years, these effective numerical methods include the proximal methods, extragradient methods, golden ratio methods, and others (for example, see \cite{Garg,Han,JOLAOSO,Konnov,Malitsky,Noor1}). However, one valuable tool stands out—the proximal operator—whose application has also been found useful in addressing numerous other numerical problems. It is important to note that the proximal operator is a strongly convex operator and hence, it guarantees a unique solution. Also, if the function $h$ is convex (which it is in our case) then MVI\eqref{MVI1} becomes a convex mixed variational inequality problem which can be best solved by the proximal type methods (see, for instance \cite{Chen,Thakur,Wang,Xia1,Xia2}), otherwise, problem \eqref{MVI1} becomes a non-convex MVI problem.
	
	\noindent Among these methods for finding an approximate solution of the MVI\eqref{MVI1}, the most popular method is the proximal type method. This method, according to Garg \textit{et al.} in \cite{Garg}, generates, when given an initial point $x(0)\in\mathcal{H},$ a trajectory $x(t)$ that approaches a point in the solution set $S^{MVI(\mathcal{T}; h; \mathcal{C})},$ and is given by
	\begin{eqnarray}\label{1.3}
		\dot{x}(t) =-\delta\big(x(t)- prox_{\lambda h}\circ(I_\mathcal{H}-\lambda\mathcal{T})(x(t))\big),
	\end{eqnarray}
	where $\dot{x}(t) = \dfrac{d}{dt}(x(t))$
	 is the derivative of the mapping $t \longmapsto x(t)$ at any time $t,$ $I_\mathcal{H}$ is the identity operator in $\mathcal{H},$ $prox_{\lambda h}(\cdot) := \underset{\underline{v}\in\mathcal{C}}{arg\min}\left\{\lambda h(\underline{v})+\dfrac{1}{2}\|\cdot -\underline{v}\|_\mathcal{H}^2\right\}$ is the proximal mapping of  $\lambda h,$ $\delta>0,$ while $\lambda>0$ is the stepsize. 
	 The author in \cite{Antipin} noted that the continuity condition of $prox_{\lambda h}\circ(I_\mathcal{H}-\lambda\mathcal{T})(x(t))$ ensures the existence of a solution of \eqref{1.3} on a finite interval. It is also important to note that, if $prox_{\lambda h}\circ(I_\mathcal{H}-\lambda\mathcal{T})(x(t))$ satisfies the Lipschitz condition, then the trajectory exists and is unique on the infinite time interval $[0,+\infty)$. This is because the operator  $prox_{\lambda h}\circ(I_\mathcal{H}-\lambda\mathcal{T})(x(t))$ is strongly convex.
	 
	 \noindent Furthermore, the  authors in \cite{Antipin,Hassan} noted that the explicit time discretization of \eqref{1.3} with respect to the time variable $t$ and without the relaxation parameter, that is, when $\delta\hbar_n=1$ gives 
	 \begin{eqnarray}\label{1.4}
	 	x_{n+1} = prox_{\lambda h} (x_n-\lambda\mathcal{T}(x_n)), \ \ \forall n\geq 0.
	 \end{eqnarray}
	And in a case of VIP, the authors in \cite{BAUS,Zhu} established  that \eqref{1.4} converges if $\mathcal{T}$ is a co-coercive operator (inverse strongly monotone), while the authors in \cite{Facchinei1,Khanh} extended the convergence of the sequence $\{x_n\}_{n\geq0}$ generated by \eqref{1.4} to the strongly (pseudo-) monotone operator $\mathcal{T}.$
	
	\noindent The extragradient method (also called the Korpelevich's extragradient method) is known to be the mostly used method in solving MVI\eqref{MVI1} whenever the operator $\mathcal{T}$ is monotone and Lipschitz continuous. In this method, two proximal evaluations are computed in each iteration.
	Although, this method was originally designed for  solving monotone MVIs (in particular, monotone VIPs) iteratively in finite dimensional spaces, however, \cite[Theorem 12.2.11]{Facchinei1} showed that the convergence of the trajectories $x(t),  \ y(t)$ (whose analog iterative sequences are given by $x_n$ and $y_n,$ respectively) is also attainable if the operator $\mathcal{T}$ is pseudomonotone. Interestingly, many researchers have further extended this method to the infinite dimensional spaces, thus enabling the reduction in the number of proximal evaluations onto the feasible set over each iteration \cite{BOT, CEN1,JOLAOSO,Khanh,Wang,Xia1}. Example of such  proximal-type method that requires lesser proximal evaluation in each evaluation is the Tseng's extragradient method. This method requires single proximal evaluation onto the feasible set over each iteration and is given by
	\begin{eqnarray}
		\begin{cases}
			y(t) = prox_{\lambda h}(x(t)-\lambda\mathcal{T}(x(t)))\\
			\dot{x}(t) = y(t) -\lambda(\mathcal{T}(y(t))-\mathcal{T}(x(t)))\\
			x(0) =x_0\\
			x_0\in\mathcal{H}, \ \ 0<\lambda<\dfrac{1}{\beta}, \ \ t\geq0,
		\end{cases}
	\end{eqnarray}
	where $\mathcal{T}:\mathcal{H}\longrightarrow\mathcal{H}$ is monotone and $\beta-$Lipschitz continuous. Recently, Abbas and Attouch \cite[Section 5.2]{Abbas} proposed the following dynamical system
	\begin{eqnarray}\label{AA}
		\begin{cases}
			\dot{x}(t)+x(t) =  prox_{\lambda h}(x(t)-\lambda\mathcal{T}(x(t)))\\
			x(0) = x_0\\
			x_0\in\mathcal{H}, \ \ \lambda>0, \ \ t\geq0.
		\end{cases}
	\end{eqnarray}
	The authors established the weak convergence of the trajectory $x(t)$ of \eqref{AA} when the operator $\mathcal{T}$ is a co-coercive operator, and the stepsize $\lambda>0$ is selected in a suitable domain bounded by the parameter of co-coercivity of the operator $\mathcal{T}.$ Similarly, Bo$\c{t}$ and Csetnek  \cite{BOT1} in their work studied the existence and uniqueness of (locally) absolutely continuous trajectories of a dynamical system of a nonexpansive operator and established the weak convergence of the orbits to the solutions set. 
	
	\noindent However, in order to improve on the computational structure of methods that solve optimization problems, Bo$\c{t}$ \textit{et al.} in \cite{BOT}  proposed the Tseng's forward-backward-forward dynamical system for solving a pseudomonotone variational inequality \eqref{1.2} and obtained a weak convergence result and the global exponential stability of the generated trajectories under the assumptions of pseudomonotonicity and strong  pseudomonotonicity of the associated Lipschitz operator. The Tseng's forward-backward-forward algorithm provides a computationally efficient alternative to Korpelevich's extragradient method for solving a convex optimization problems on closed sets. As mentioned earlier, Tseng's algorithm requires only a single projection operation per iteration, thereby reducing computational complexity. This advantage makes Tseng's algorithm an attractive choice for addressing MVIPs governed by monotone and Lipschitz continuous operators.
	
	\noindent Motivated by the works in the literature, we propose a dynamical system for finding a solution of the MVI\eqref{MVI1} with the following contributions:
	\begin{itemize}
		\item Our proposed system features a single proximal evaluation (projection evaluation in a case where $h$ is the indicator function) in each step;
		\item We obtain a weak convergence result for the generated orbits to a solution of MVI under the general monotonicity condition;
		\item Under the strong pseudomonotonicity, we establish the global exponential stability of the equilibrium point of the system studied;
		\item  Numerical examples are provided to demonstrate the convergence of the generated trajectories to the equilibrium point of the system under consideration.
	\end{itemize}
	 
	 \noindent The rest of the content of this paper is arranged as follows:  In Section \textbf{\ref{pre}}, we present some important definitions, lemmas and preliminary
	 results subsequently needed in this work. In Section \textbf{\ref{MAIN RESULT}}, we present and discuss our proposed continuous system with both weak convergence and global exponential stability given. Section \textbf{\ref{numerics}} presents some numerical results which serve as practical illustrations of the convergence of the generated orbits to the equilibrium point of our system. 
	 
	\section{Preliminaries}\label{pre}
	\noindent In this section, we present some basic definitions and lemmas that are necessary for the convergence analysis of our proposed system. While the notations $\rightharpoonup$ and $\rightarrow$  will be adapt to mean weakly and strongly convergence, respectively. We shall denote $\overline{\mathbb{R}}:=\mathbb{R}\cup\{-\infty, +\infty\}$ to be the extended real number.\\
	
\begin{definition}\label{def2.1}
	A function $h:[0,b]\longrightarrow\mathcal{H}$ (where $0<b<+\infty$) is said to be absolutely continuous if one of the following properties is true:
	\begin{enumerate}
		\item [(1)] there is an integrable function $\phi:[0,b]\longrightarrow\mathcal{H}$ such that 
		\begin{eqnarray}\label{form1}
			h(t)=h(0)+\int_{0}^{t}\phi(x)dx \ \ \ \forall t\in[0,b];
		\end{eqnarray}
		\item [(ii)] $h$ is continuous and its distributional derivative belongs to the  Lebesgue space $L^1([0,b]; \mathcal{H}).$ That is, the distributional derivative of $h$ is Lebesgue integrable on $[0,b];$ 
		\item [(iii)] for any $\epsilon>0,$ there is $\delta>0$ such that for every finite family of intervals $I_j=(a_j,b_j),$ the following holds
		\begin{eqnarray*}
			\left(I_i\cap I_j = \emptyset ~ \mbox{for} ~ i\neq j~ \mbox{and} ~ \sum_{j}\vert b_j-a_j\vert<\delta\right)\implies\sum_{j}\|h(b_j)-h(a_j)\|<\epsilon.
		\end{eqnarray*}
	\end{enumerate}
\end{definition}
\begin{remark}\label{remark continuous}\cite{Banert,BOT1}\noindent
	\begin{enumerate}
		\item [(a)] It indeed ensues immediately from the definition above that whenever a function is absolutely continuous then, it is differentiable almost everywhere, and in this situation, its distributional derivatives  coincide with its derivative, hence, by the integration formula \eqref{form1}, we can easily recover the function from its derivative $h' = \phi.$ 
		\item [(b)] By considering the characterization in Definition \ref{def2.1}(iii), it can be easily verified that if $h:[0,b]\longrightarrow\mathcal{H},$ where $b>0,$ is absolutely continuous and $\mathcal{G}:\mathcal{H}\longrightarrow\mathcal{H}$ is $\beta-$Lipschitz continuous with $\beta>0,$ then the function $f=\mathcal{G}\circ h$ is absolutely continuous. Moreover, $f$ is almost everywhere differentiable and the inequality $\|f'(\cdot)\|\leq \beta\|h'(\cdot)\|$ is satisfied almost everywhere.
	\end{enumerate}
\end{remark}
\begin{definition} \cite{Voung 1}\label{defn expo}
	Let $\mathcal{T}:\mathcal{H}\longrightarrow\mathcal{H}$ be a Lipschitz continuous operator, $\omega : [0, +\infty) \longrightarrow\mathcal{H}$ such that  $t\longmapsto \omega(t)$ is a differentiable function. Let $\mathcal{B}(\bar{c}, r)$ denote the open ball centered at $\bar{c}$ and has the radius $r.$ Let 
	\begin{eqnarray}\label{expo}
		\dot{\omega}(t)=\beta\mathcal{T}(\omega(t)), \ \ \ t\geq0,
	\end{eqnarray}
	where $\beta\in\mathbb{R},$ denote the general dynamical system.
		Then, 
		\begin{enumerate}
			\item [I.] a point $\bar{c}\in\mathcal{H}$ is called an equilibrium point of the general dynamical system \eqref{expo} if $\mathcal{T}(\bar{c})=0;$
			\item [II.] an equilibrium point $\bar{c}$ of \eqref{expo} is said to be stable if, for any $\xi>0,$ there is $\delta>0$ such that, for all $\omega(0)\in\mathcal{B}(\bar{c}, \delta),$ the solution $\omega(t)$ of the continuous system exists and is contained in $\mathcal{B}(\bar{c}, \xi)$ for every $t>0.$ That is, $\omega(t)\in \mathcal{B}(\bar{c}, \xi), \ \forall t>0;$
			\item [III.] a stable equilibrium point $\bar{c}$ of \eqref{expo} is said to be asymptotically stable if there exists $\delta>0$ such that for every solution $\omega(t)$ of \eqref{expo} with $\omega(0)\in\mathcal{B}(\bar{c}, \delta),$ it follows that 
			$$\underset{t\to+\infty}{\lim}\omega(t) = \bar{c};$$  
			\item [IV.] an equilibrium point $\bar{c}$ of \eqref{expo} is said to be  exponentially stable if there exist $\delta > 0$ and
			constants $k>0$ and $\alpha>0$ such that, for any solution $\omega(t)$ of \eqref{expo} with $\omega(0)\in\mathcal{B}(\bar{c}, \delta),$ one has  
			\begin{eqnarray}\label{GES}
				\|\omega(t)-\bar{c}\|\leq k\|\omega(0)-\bar{c}\|e^{-\alpha t}, \ \ \ \forall t\in[0, +\infty).
			\end{eqnarray}
			In addition, if \eqref{GES} is satisfied for all solutions of \eqref{expo}, then the equilibrium point $\bar{c}$ is said to be globally exponentially stable (GES).
		\end{enumerate}
\end{definition}

\begin{definition}
	Let $h:\mathcal{H}\longrightarrow\overline{\mathbb{R}}$ be a real-valued function. Then
	\begin{enumerate}
		\item [(i)] the effective domain of $h$ is defined by $dom~h:=\{\underline{u}\in\mathcal{H} \ \vert \ h(\underline{u})<+\infty\};$
		\item [(ii)] $h$ is called a proper function if $h(\underline{u})>-\infty$ for any $\underline{u}\in \mathcal{H}$ and $dom~h\neq\emptyset.$ Thus, one can also verify that  $h(\underline{u})=+\infty$ for any $\underline{u}\notin domh~h;$
		\item [(iii)] $h$ is called a convex function if the domain of $h$ is convex and for any $\underline{u},\underline{v}\in dom~h,$ we have
		\begin{eqnarray*}
			h(\alpha\underline{u}+(1-\alpha)\underline{v})\leq\alpha h(\underline{u}) + (1-\alpha)h(\underline{v}), \ \ \ \forall\alpha\in[0,1].
		\end{eqnarray*}
	\end{enumerate}
\end{definition}	
	
	\begin{definition}
		Let $h:\mathcal{C}\subset\mathcal{H}\longrightarrow\overline{\mathbb{R}}$ be a proper, convex and lower semicontinuous real-valued function, where $\mathcal{C}$ is a nonempty closed convex set. For each $\underline{u}, \underline{v} \in\mathcal{C},$ the mapping $\mathcal{T} : \mathcal{H}\longrightarrow\mathcal{H}$ is said to be 
		\begin{enumerate}
			\item [(i)] Lipschitz continuous if there exists a constant $\beta>0,$ called the modulus (or Lipschitz constant), such that
			\begin{eqnarray*}
				\|\mathcal{T}\underline{u}-\mathcal{T}\underline{v}\|\leq\beta\|\underline{u}-\underline{v}\|;
			\end{eqnarray*} 
			\item [(ii)] monotone if 
			\begin{eqnarray*}
			\langle \mathcal{T}\underline{u}-\mathcal{T}\underline{v}, \underline{u}-\underline{v}\rangle\geq0;	
			\end{eqnarray*}
			\item [(iii)] strongly monotone if there exists a constant $\alpha>0$ such that
			\begin{eqnarray*}
				\langle \mathcal{T}\underline{u}-\mathcal{T}\underline{v}, \underline{u}-\underline{v}\rangle\geq\alpha\|\underline{u}-\underline{v}\|^2;
			\end{eqnarray*}
			\item [(iv)] pseudomonotone if  
			\begin{eqnarray*}
				\langle\mathcal{T}\underline{u}, \underline{v}-\underline{u}\rangle\geq0 \implies \langle\mathcal{T}\underline{v}, \underline{v}-\underline{u}\rangle\geq0;
			\end{eqnarray*}
			\item [(v)] $h-$pseudomonotone on $\mathcal{C},$ if 
			\begin{eqnarray*}
				\langle\mathcal{T}\underline{u}, \underline{v}-\underline{u}\rangle +h(\underline{v}) - h(\underline{u})\geq0 \implies \langle\mathcal{T}\underline{v}, \underline{v}-\underline{u}\rangle +h(\underline{v}) - h(\underline{u})\geq0 \ \  \mbox{(See \cite{JOLAOSO})};
			\end{eqnarray*}
			\item [(vi)] $h-$strongly pseudomonotone on $\mathcal{C},$ if there exists a constant $\mu>0$ such that the following holds,
			\begin{eqnarray*}
				\langle\mathcal{T}\underline{u}, \underline{v}-\underline{u}\rangle +h(\underline{v}) - h(\underline{u})\geq0 \implies \langle\mathcal{T}\underline{v}, \underline{v}-\underline{u}\rangle +h(\underline{v}) - h(\underline{u})\geq \mu\|\underline{u}-\underline{v}\|^2 \  \  \mbox{(See \cite{Zheng})}.
			\end{eqnarray*} 
		\end{enumerate}
	\end{definition}
	\begin{remark}\noindent
		\begin{enumerate}
			\item One can easily verify that $\alpha\leq\beta$ if $\mathcal{T}$ is $\alpha-$strongly monotone and $\beta-$Lipschitz continuous on $\mathcal{C}.$
			\item Notice that the $h-$pseudomonotonicity of $\mathcal{T}$ on the $dom~h$ does not necessary mean that $\mathcal{T}$ is pseudomonotone. This can be easily checked using Example \ref{ex1}, but if $h=\delta_\mathcal{C},$ then the $0-$pseudomonotonicity of $\mathcal{T}$ corresponds to the pseudomonotonicity condition of $\mathcal{T}.$
		\end{enumerate}
	\end{remark}
	\begin{lemma}\cite{BAUS}\label{A1}
		For each $\underline{u}\in\mathcal{H},$ $\underline{v}\in\mathcal{C}$ and a positive constant $\lambda,$ the following inequality is true\\
		$$\lambda(h(\underline{v})-h(prox_{\lambda h}(\underline{u})))\geq \langle \underline{u}- prox_{\lambda h}(\underline{u}), \underline{v}-prox_{\lambda h}(\underline{u})\rangle,$$ \\
		where $prox_{\lambda h}(\underline{u}) := \underset{\underline{v}\in\mathcal{C}}{argmin}\left\{\lambda h(\underline{v})+\frac{\|\underline{u}-\underline{v}\|^2}{2}\right\}.$
	\end{lemma}
	\begin{lemma}\label{BOT2.1}\cite[Lemma 5.2]{Abbas}\cite[Lemma 2.1]{BOT} (continuous version of a result which states the convergence of quasi-Fej$\acute{e}$r monotone sequences)\noindent\\
		If $1\leq p<\infty,$ $1\leq r<\infty,$ $\phi:[0, +\infty)\longrightarrow [0, +\infty)$ is locally absolutely continuous, $\phi\in L^p([0, +\infty)),$ $\omega: [0, +\infty)\longrightarrow \mathbb{R},$ $\omega \in L^r([0, +\infty))$ and for almost each $t\in[0, +\infty)$
		\begin{eqnarray*}
			\dfrac{d}{dt}\phi(t)\leq\omega(t),
		\end{eqnarray*}
		then, $\underset{t\to+\infty}{\lim}\phi(t) = 0.$
	\end{lemma}
   \begin{lemma}\label{BOT2.2}\cite[Lemma 5.3]{Abbas}\cite[Lemma 2.2]{BOT} (continuous version of the Opial Lemma)\noindent\\
	Let $\mathcal{C}\subset\mathcal{H}$ be a nonempty subset of $\mathcal{H}$ and $x:[0,+\infty)\longrightarrow\mathcal{H}$ a given map.\\
	Assume that
	\begin{enumerate}
		\item [(i)] the $\underset{t\to+\infty}{\lim}\|x(t)-\bar{x}\|$ exists for each $\bar{x}\in\mathcal{C};$
		\item [(ii)] every weak sequential cluster point of the map $x$ belongs to $\mathcal{C}.$\\
		Then, there exists an element $x_\infty\in\mathcal{C}$ such that $x(t)$ converges weakly to $x_\infty$ as $t\to+\infty.$
	\end{enumerate}
   \end{lemma}
\begin{lemma}\label{Gronwall}\cite[Lemma 1.1]{Ke Li} (Gronwall's lemma):
	Let $\varphi:[0,+\infty)\longrightarrow[0,+\infty)$ be an absolutely nonnegative continuous function that satisfies almost everywhere the differential inequality
	\begin{eqnarray*}
		\dot{\varphi}(t)\leq -\alpha\varphi(t) + K \ \ \mbox{for some} \  \alpha>0, \ \ K>0, \ t\in(0, +\infty),
	\end{eqnarray*}
	then
	\begin{eqnarray*}
		\varphi(t)\leq\varphi(0)e^{-\alpha t} + \dfrac{K}{\alpha}. 
	\end{eqnarray*}
\end{lemma}

		 
\begin{lemma}\label{prox}
	Suppose $\mathcal{T}:\mathcal{H}\longrightarrow\mathcal{H}$ is Lipschitz continuous with constant $\beta>0$ and $h:\mathcal{H}\longrightarrow\mathbb{R}$ is a proper convex lower semicontinuous real-valued function. Then the operator $prox_{\lambda h}\circ(I_\mathcal{H}-\lambda\mathcal{T}):\mathcal{H}\longrightarrow\mathbb{R}$ is $(1+\lambda\beta)-$Lipschitz continuous, where $I_\mathcal{H}$ is the identity operator in $\mathcal{H},$ $\lambda>0.$ 
\end{lemma}
\begin{proof}\noindent\\
	Since $prox_{\lambda h}(\cdot)$ is nonexpansive, we have
	$$\|prox_{\lambda h}(\underline{u})-prox_{\lambda g}(\underline{v})\|\leq\|\underline{u}-\underline{v}\|.$$
	This implies that
	\begin{eqnarray*}
		\|prox_{\lambda h}\circ(I_\mathcal{H}-\lambda\mathcal{T})x(t)-prox_{\lambda h}\circ(I_\mathcal{H}-\lambda\mathcal{T})y(t)\|&\leq&\|(I_\mathcal{H}-\lambda\mathcal{T})x(t)-(I_\mathcal{H}-\lambda\mathcal{T})y(t)\| \\
		&=&\|x(t)-y(t)+\lambda\left[\mathcal{T}(y(t))-\mathcal{T}(x(t))\right]\|\\
		&\leq& \|x(t)-y(t)\|+\lambda\beta\|x(t)-y(t)\|\\
		&=& (1+\lambda\beta)\|x(t)-y(t)\|.
	\end{eqnarray*}
\end{proof}
	\section{Main Results}\label{MAIN RESULT}
\noindent In this section, we present and analyze a forward-backward-forward dynamical system for solving  MVI \eqref{MVI1}. Before we present our system, let us first consider some assumptions which the convergence of the generated trajectories is based on.
\begin{assup}\label{ASS1}
	Let $\mathcal{T}:\mathcal{H}\longrightarrow\mathcal{H}$ be an operator and $h:\mathcal{C}\subset\mathcal{H}\longrightarrow\overline{\mathbb{R}},$ a real-valued function, where the set $\mathcal{C}$ is nonempty, closed and convex.
	Suppose the following hold:
	\begin{enumerate}
		\item [(C$_1$)] the solutions set, $S^{MVI(\mathcal{T}; h; \mathcal{C})}\neq\emptyset;$
		\item [(C$_2$)] $\mathcal{T}$ is Lipschitz continuous with Lipschitz constant $\beta>0;$
		\item [(C$_3$)] $h$ is a proper, lower semicontinuous and convex real-valued function on $\mathcal{C}$;
		\item [(C$_4$)] $\mathcal{T}$ is monotone. In particular, $\mathcal{T}$ and $h$ satisfy the following general monotonicity condition on $\mathcal{C}$
		\begin{eqnarray}
			\langle \mathcal{T}(\underline{v}), \underline{v}-\underline{u}\rangle +h(\underline{v})-h(\underline{u})\geq0, \ \ \forall \underline{v}\in\mathcal{C}, \ \ \forall \underline{u}\in S^{MVI(\mathcal{T}; h; \mathcal{C})};
		\end{eqnarray}
		\item [(C$_5$)] $\mathcal{T}$ is sequentially weakly continuous on $\mathcal{C},$ that is, if $x(t_n)\rightharpoonup\bar{x},$ then $\mathcal{T}(x(t_n))\rightharpoonup\mathcal{T}(\bar{x}).$
	\end{enumerate}
\end{assup}

\begin{remark}\noindent
	\begin{itemize}
		\item [(i)]  Notice that the operator $\mathcal{T}$ satisfies assumption (C$_4$) if it is $h-$pseudomonotone  (monotone) on $\mathcal{C}.$
		\item [(ii)]  It is also reasonable to mention that the convex function $h$ can be Lipschitz continuous. Though this condition is strong but it  does not in anyway exclude the VIP, since the indicator function over the set $\mathcal{C}$ is Lipschitz continuous on its effective field.
	\end{itemize}

 \end{remark}
\hrule
Dynamical System for Solving Mixed Variational Inequality Problem \eqref{MVI1}.
\hrule
\noindent Given $\lambda>0$ and $x(0)\in\mathcal{H},$ the trajectories generated by our dynamical system are given as follows
\begin{eqnarray}\label{DS MVI}
	\begin{cases}
		y(t) = prox_{\lambda h}(x(t)-\lambda\mathcal{T}x(t))\\
		\dot{x}(t)+x(t)=y(t)+\lambda\left[\mathcal{T}(x(t))-\mathcal{T}(y(t))\right]\\
		x(0)= x_0.
	\end{cases}
\end{eqnarray}
\begin{remark}\cite[Defintion 2]{Banert}, \cite[Definition 2]{BOT1} 
	We say that the trajectory $x : [0, +\infty) \longrightarrow \mathcal{H}$ is a strong global solution of \eqref{DS MVI} if the
	following hold:
	\begin{enumerate}
		\item [I.]  the trajectory $x : [0, +\infty) \longrightarrow\mathcal{H}$  is locally absolutely continuous, that is, absolutely continuous on
		every interval $[0, \rho], \ \  0 < \rho < +\infty;$
		\item [II.] $\dot{x}(t)+x(t)-y(t)-\lambda\left[\mathcal{T}(x(t))-\mathcal{T}(y(t))\right] = 0$ for almost each $t\in[0,+\infty);$
		\item [III.] $x(0) = x_0.$
	\end{enumerate}
\end{remark}
	
\begin{lemma}\label{DS1}
		Assume that C$_1$---C$_4$ of Assumption \ref{ASS1} are satisfied by the operator $\mathcal{T}$ and the real-valued function $h.$ Then, the following holds for each $\bar{x}\in S^{MVI(\mathcal{T}; h; \mathcal{C})}$ 
		\begin{eqnarray*}
			\langle \dot{x}(t), x(t)-\bar{x}\rangle\leq -[1-\lambda(1+\beta^2)]\|x(t)-y(t)\|^2 \ \ \forall t\in[0,+\infty).
		\end{eqnarray*}
	\end{lemma}
	\begin{proof}\noindent\\
		Since C$_1$ of Assumption \ref{ASS1} holds, that is, $S^{MVI(\mathcal{T}; h; \mathcal{C})}\neq\emptyset.$ Let $\bar{x}\in S^{MVI(\mathcal{T}; h; \mathcal{C})},$ then for every $y(t)\in\mathcal{C},$ it follows that
		\begin{eqnarray*}
			\langle \mathcal{T}\bar{x}, y(t)-\bar{x}\rangle+h(y(t))-h(\bar{x})\geq 0 \ \ \ \  \forall t\in[0,+\infty).
		\end{eqnarray*}
		Since $\mathcal{T}$ and $h$ satisfy the general monotonicity condition on $\mathcal{C}.$ That is, C$_4$ of Assumption \ref{ASS1} is satisfied, we have
		\begin{eqnarray}\label{1}
			\langle\mathcal{T}(y(t)), y(t)-\bar{x}\rangle+h(y(t))-h(\bar{x})\geq0 \ \ \forall y(t)\in\mathcal{C}, \ \ \forall \bar{x}\in S^{MVI(\mathcal{T}; h; \mathcal{C})}, \ \ t\in[0,+\infty).
		\end{eqnarray}
		Using the definition of $y(t)= prox_{\lambda h}(x(t)-\lambda\mathcal{T}(x(t))),$ then by Lemma \ref{A1} we have that
		\begin{eqnarray}\label{2}
			&&\langle x(t)-\lambda\mathcal{T}(x(t))-y(t), \bar{x}-y(t)\rangle\nonumber\\
			&&\leq\lambda\left[h(\bar{x})-h(y(t))\right]  \ \ \forall y(t)\in\mathcal{C},  \ \forall \bar{x}\in S^{MVI(\mathcal{T}; h; \mathcal{C})},  \ t\in[0,+\infty).
		\end{eqnarray} 
		From \eqref{1}, we have
		\begin{eqnarray}\label{3}
			\langle\lambda\mathcal{T}(y(t)), \bar{x}-y(t)\rangle\leq\lambda\left[h(y(t))-h(\bar{x})\right].
		\end{eqnarray}
		Combining \eqref{2} and \eqref{3}, we obtain
		\begin{eqnarray*}
			\langle x(t)-y(t)-\lambda\left[\mathcal{T}(x(t))-\mathcal{T}(y(t))\right], y(t)-\bar{x}\rangle \geq 0.
		\end{eqnarray*}
		By the definition of $\dot{x}(t)$ in  dynamical system \eqref{DS MVI}, we have from the last inequality that
		\begin{eqnarray}\label{3a}
			0&\leq&\langle x(t)-y(t)-\lambda \left[\mathcal{T}(x(t))-\mathcal{T}(y(t))\right], y(t)-\bar{x}\rangle\nonumber\\
			&=&\langle x(t)-y(t)-\lambda \left[\mathcal{T}(x(t))-\mathcal{T}(y(t))\right], y(t)-x(t)\rangle\nonumber\\
			&&\;+\; \langle x(t)-y(t)-\lambda \left[\mathcal{T}(x(t))-\mathcal{T}(y(t))\right], x(t)-\bar{x}\rangle\nonumber\\
			&=& \langle x(t)-y(t)-\lambda \left[\mathcal{T}(x(t))-\mathcal{T}(y(t))\right], y(t)-x(t)\rangle -\langle \dot{x}(t), x(t)-\bar{x}\rangle.
		\end{eqnarray}
		Using the fact that $\mathcal{T}$ is $\beta-$Lipschitz continuous, then \eqref{3a} gives us that
		\begin{eqnarray}\label{4}
			\langle \dot{x}(t), x(t)-\bar{x}\rangle&\leq& \langle x(t)-y(t)-\lambda \left[\mathcal{T}(x(t))-\mathcal{T}(y(t))\right], y(t)-x(t)\rangle\nonumber\\
			&=& -\|x(t)-y(t)\|^2+\lambda \langle\mathcal{T}(x(t))-\mathcal{T}(y(t)), x(t)-y(t)\rangle\nonumber\\
			&\leq&-\|x(t)-y(t)\|^2 +\dfrac{\lambda}{2}(1+\beta^2)\|x(t)-y(t)\|^2\nonumber\\
			&\leq& -\|x(t)-y(t)\|^2+\lambda(1+\beta^2)\|x(t)-y(t)\|^2\nonumber\\
			&=& -\left[1-\lambda(1+\beta^2)\right]\|x(t)-y(t)\|^2 \ \ \ \ \forall t\in [0,+\infty).
		\end{eqnarray}
	\end{proof}
	\begin{lemma}\label{Lemma bdd}
		Assume that $C_1-C_4$ hold,
		 and $0<\lambda<\dfrac{1}{1+\beta^2}.$ Then, for each $\bar{x}\in S^{MVI(\mathcal{T}; h; \mathcal{C})},$ the function $t\longmapsto\|x(t)-\bar{x}\|^2$ is nonincreasing, that is, the trajectory $x(t)$ is bounded and convergent. Consequently,
		\begin{eqnarray*}
			\int_{0}^{+\infty}\|x(t)-y(t)\|^2 dt < +\infty \ \ \mbox{and}\ \ \underset{t\to+\infty}{\lim}\|x(t)-y(t)\| = 0.
		\end{eqnarray*}
	\end{lemma}
	\begin{proof}\noindent\\
		Recall from Lyapunov analysis that
		\begin{eqnarray}\label{4b}
			\dfrac{d}{dt}\|x(t)-\bar{x}\|^2 = 2\langle\dot{x}(t), x(t)-\bar{x}\rangle.
		\end{eqnarray}
		Using \eqref{4} in \eqref{4b}, we obtain
		\begin{eqnarray}\label{5}
			\dfrac{d}{dt}\|x(t)-\bar{x}\|^2\leq -2\left[1-\lambda(1+\beta^2)\right]\|x(t)-y(t)\|^2 \ \ \forall t\in[0,+\infty).
		\end{eqnarray}
		Since $0<\lambda<\dfrac{1}{1+\beta^2},$ it follows that
		$\dfrac{d}{dt}\|x(t)-\bar{x}\|^2\leq0 \ \ \forall t\in[0,+\infty).$ Hence, $t\longmapsto\|x(t)-\bar{x}\|^2$ is a nonincreasing function for each $ t\in[0,+\infty).$ Thus the trajectory $x(t)$ is bounded.\\
		
		\noindent On the other hand, from \eqref{5}, we obtain
		\begin{eqnarray}\label{6}
		2\left[1-\lambda(1+\beta^2)\right]\|x(t)-y(t)\|^2dt\leq-d\left[\|x(t)-\bar{x}\|^2 \right].	
		\end{eqnarray}
		Integrating \eqref{6} from $0$ to $T,$ where $T>0,$ gives
		\begin{eqnarray}\label{7}
			2\left[1-\lambda(1+\beta^2)\right]\int_{0}^{T}\|x(t)-y(t)\|^2dt&\leq&\|x(0)-\bar{x}\|^2-\|x(T)-\bar{x}\|^2\nonumber\\
			&\leq& \|x(0)-\bar{x}\|^2.
		\end{eqnarray}
		Letting $T\longrightarrow+\infty$ in \eqref{7}, then
		\begin{eqnarray*}
			\int_{0}^{+\infty}\|x(t)-y(t)\|^2dt\leq\dfrac{1}{2}\dfrac{1}{1-\lambda(1+\beta^2)}\|x_0-\bar{x}\|^2<+\infty.
		\end{eqnarray*} 
		Next, we show that $\underset{t\to+\infty}{\lim}\|x(t)-y(t)\| = 0.$\\
		Note that $prox_{\lambda h}$ is nonexpansive and $\mathcal{T}$ is $\beta-$Lipschitz continuous. Recall from Lemma \ref{prox} that $prox_{\lambda h}\circ(I_\mathcal{H}-\lambda\mathcal{T})$ is Lipschitz continuous with constant $(1+\lambda\beta).$ Using the definition of $y(t),$ that is
		$$y(t)=prox_{\lambda h}\circ(I_\mathcal{H}-\lambda\mathcal{T})(x(t)),$$ 
		we have that the trajectory $y(t)$ is locally absolutely continuous, hence by Remark \ref{remark continuous}b, it follows for almost all $t\geq0,$ that 
		\begin{eqnarray}\label{8}
			\|\dot{y}(t)\|\leq(1+\lambda \beta)\|\dot{x}(t)\|.
		\end{eqnarray}
		On the other hand, using the definition of $\dot{x}(t),$ we have
		\begin{eqnarray}\label{9}
			\|\dot{x}(t)\| &=& \|y(t)-x(t)+\lambda\left[\mathcal{T}(x(t))-\mathcal{T}(y(t))\right]\nonumber\\
			&\leq&\|x(t)-y(t)\|+\lambda\|\mathcal{T}(x(t))-\mathcal{T}(y(t))\|\nonumber\\
			&\leq&(1+\lambda\beta)\|x(t)-y(t)\| \ \ \ \ \ \forall t\in[0,\infty).
		\end{eqnarray}
		Therefore, using the Lyapunov analysis, Cauchy inequality, and inequalities \eqref{8} and \eqref{9}, we obtain for almost every $t\in[0,+\infty)$ that
		\begin{eqnarray*}
			\dfrac{d}{dt}\|x(t)-y(t)\|^2 &=& 2\langle\dot{x}(t)-\dot{y}(t), x(t)-y(t)\rangle\\
			&\leq&2\|\dot{x}(t)-\dot{y}(t)\|\|x(t)-y(t)\|\\
			&\leq&2(\|\dot{x}(t)\|+\|\dot{y}(t)\|)\|x(t)-y(t)\|\\
			&\leq&2\left[(1+\lambda\beta)\|x(t)-y(t)\|+(1+\lambda\beta)^2\|x(t)-y(t)\|\right]\|x(t)-y(t)\|\\
			&=& 2\left[(1+\lambda\beta)(2+\lambda\beta)\right]\|x(t)-y(t)\|^2.
		\end{eqnarray*}
		Hence, we conclude from Lemma \ref{BOT2.1} that
		$$\underset{t\to+\infty}{\lim}\|x(t)-y(t)\|=0.$$
	\end{proof}

    \begin{theorem}
		Suppose that Assumption \ref{ASS1} is satisfied and $0<\lambda<\dfrac{1}{1+\beta^2}.$ Then, the trajectories $x(t)$ and $y(t)$ generated by the dynamical system \eqref{DS MVI} converges weakly to a point in $S^{MVI(\mathcal{T}; h; \mathcal{C})},$ as $t\to+\infty.$
	\end{theorem}
	\begin{proof}\noindent\\
		Let $\bar{c}\in\mathcal{H}$ be a weak sequential cluster point of $x(t)$ as $t\to+\infty.$ For each $n\geq0,$ let $\{t_n\}_{n\geq0}$ be a sequence in $[0, +\infty)$ with $t_n\to+\infty$ and $x(t_n)\rightharpoonup\bar{c},$ as $n\to+\infty.$
		
		\noindent Since $\underset{t\to+\infty}{\lim}\|x(t)-y(t)\|=0,$ we also have that $y(t_n)\rightharpoonup\bar{c}$ as $n\to+\infty.$ Indeed, $\|\mathcal{T}(x(t_n))-\mathcal{T}(y(t_n))\|\to 0$ as $n\to+\infty,$ since $\mathcal{T}$ is Lipschitz continuous.\\
		
		\noindent We now show that $\bar{c}\in S^{MVI(\mathcal{T}; h; \mathcal{C})}.$ 
		Since the sequence $\{y(t_n)\}_{n\geq0}\subset\mathcal{C}$ and $\mathcal{C}$ is weakly closed, we have that $\bar{c}\in\mathcal{C}.$\\
		Let $y\in \mathcal{C}$ be fixed. For each $n\geq0,$ we have 
		$$y(t_n) = prox_{\lambda h}(x(t_n)-\lambda\mathcal{T}(x(t_n))).$$
		Therefore, by \eqref{2}, it follows that
		\begin{eqnarray*}
			\langle x(t_n)-\lambda\mathcal{T}(x(t_n))-y(t_n), y-y(t_n)\rangle\leq\lambda\left[h(y)-h(y(t_n))\right].
		\end{eqnarray*}
		Or equivalently
		\begin{eqnarray}\label{10}
		\lambda\left[h(y(t_n))-h(y)\right]&\leq& \langle y(t_n)-x(t_n)+\lambda\mathcal{T}(x(t_n)), y-y(t_n)\rangle	\nonumber\\
		&=& \langle y(t_n)-x(t_n), y-y(t_n)\rangle +\lambda \langle\mathcal{T}(x(t_n)), y-y(t_n)\rangle \nonumber\\
		&=& \langle y(t_n)-x(t_n), y-y(t_n)\rangle +\lambda \langle\mathcal{T}(x(t_n)), y-x(t_n)\rangle+ \lambda\langle\mathcal{T}(x(t_n)), x(t_n)-y(t_n)\rangle\nonumber\\
		&\leq& \langle y(t_n)-x(t_n), y-y(t_n)\rangle +\lambda \langle\mathcal{T}y, y-x(t_n)\rangle+ \lambda\langle\mathcal{T}(x(t_n)), x(t_n)-y(t_n)\rangle,
		\end{eqnarray}
		where the last inequality is obtained using the fact that $\mathcal{T}$ is a monotone  operator ($\mathcal{T}$ satisfies the general monotonicity condition).
		Letting $n\longrightarrow+\infty$ in \eqref{10} (if necessary, the sub trajectories), noting that $\|x(t_n)-y(t_n)\|\longrightarrow 0,$  $\mathcal{T}$ is sequentially weakly continuous on $\mathcal{C},$  $h$ is a lower semicontinuous function and $\{y(t_n)\}_{n\geq 0}$ is bounded, therefore, we have
		\begin{eqnarray}\label{11}
			\langle \mathcal{T}y, y-\bar{c}\rangle +h(y)-h(\bar{c})\geq0.
		\end{eqnarray}
		Given any arbitrary point $z\in\mathcal{C},$ let $y_\sigma = \sigma z+(1-\sigma)\bar{c},$ where $\sigma\in(0,1), \ \bar{c}\in\mathcal{C}.$ Then, $y_\sigma\in\mathcal{C},$ since $\mathcal{C}$ is convex.\\
		
		\noindent Now, put $y=y_\sigma$ in \eqref{11}, we obtain
		\begin{eqnarray}\label{12}
			\langle \mathcal{T}y_\sigma, y_\sigma-\bar{c}\rangle +h(y_\sigma)-h(\bar{c})\geq0.
		\end{eqnarray}
		Hence, since $h$ is convex, we have
		\begin{eqnarray}\label{13}
			\langle \mathcal{T}y_\sigma, z-\bar{c}\rangle +h(z)-h(\bar{c})\geq0.
		\end{eqnarray}
		Taking limit as $\sigma\longrightarrow0$ in \eqref{13} and using the fact that $\mathcal{T}$ is Lipschitz, we obtain
		\begin{eqnarray*}
			\langle \mathcal{T}\bar{c}, z-\bar{c}\rangle +h(z)-h(\bar{c})\geq0.
		\end{eqnarray*}
		Since $z\in\mathcal{C}$ is arbitrarily chosen, we have that  $\bar{c}\in S^{MVI(\mathcal{T}; h; \mathcal{C})}.$\\
		
		\noindent On the other hand, by Lemma \ref{Lemma bdd},  we have that $\underset{t\to+\infty}{\lim}\|x(t)-\bar{c}\|$ exists for each  $\bar{c}\in S^{MVI(\mathcal{T}; h; \mathcal{C})}$. Hence, by Lemma \ref{BOT2.2}, the trajectory $x(t)$ converges weakly to an element in the solution set,  $S^{MVI(\mathcal{T}; h; \mathcal{C})}$ as $t\longrightarrow+\infty.$ Again, by Lemma \ref{BOT2.2}, we have that
		\begin{eqnarray*}
			\underset{t\to+\infty}{\lim}\|x(t)-y(t)\| = 0,
		\end{eqnarray*}
		we immediately have that the trajectory $y(t)$ also converges weakly to the same element in the solution set,  $S^{MVI(\mathcal{T}; h; \mathcal{C})},$ as $t\longrightarrow+\infty.$
	\end{proof}

\hfill
    
	\begin{theorem}
		Assume that $\mathcal{T}:\mathcal{H}\longrightarrow\mathcal{H}$ is $h-$strongly pseudomonotone on $\mathcal{C}$ with modulus $\mu>0$ and Lipschitz continuous with constant $\beta>0,$ and $0<\lambda<\dfrac{1}{1+\beta^2}.$ Assume that $h$ is a proper, lower semicontinuous convex real-valued function. Then for each $t\in(0,+\infty),$ the unique solution  $\bar{c}$  of the MVI \eqref{MVI1}, which is also the equilibrium point of the dynamical system \eqref{DS MVI}, is globally exponentially stable for all the solutions $x(t).$\\
		That is, 
		\begin{eqnarray*}
			\|x(t)-\bar{c}\|^2\leq\|x(0)-\bar{c}\|^2e^{-\alpha t} \ \ \forall t\in[0,+\infty),
		\end{eqnarray*}
		where $\alpha=2\left[1-\lambda(1+\beta^2)\right]\left(\dfrac{\lambda\mu}{1+\lambda(\mu+\beta)}\right)^2.$
	\end{theorem}
	\begin{proof}\noindent\\
		Let $\bar{c}\in\mathcal{C}$ be a solution to MVI \eqref{MVI1} be fixed for all $t\in(0,+\infty).$ Since $y(t)\in\mathcal{C},$ we obtain
		\begin{eqnarray}\label{T21}
			\langle\mathcal{T}(\bar{c}), y(t)-\bar{c}\rangle+h(y(t))-h(\bar{c})\geq0.
		\end{eqnarray}
		By the $h-$strong pseudomonotonicity of $\mathcal{T}$ on $\mathcal{C},$ we have
		\begin{eqnarray}\label{T22}
			\langle\mathcal{T}(y(t)), y(t)-\bar{c}\rangle+h(y(t))-h(\bar{c})\geq\mu\|y(t)-\bar{c}\|^2.
		\end{eqnarray}
		On the other hand, using the Lipschitz continuity of $\mathcal{T}$ and inequality \eqref{T22}, we obtain that
		\begin{eqnarray*}
			\langle\mathcal{T}(x(t)), \bar{c}-y(t)\rangle &=& \langle\mathcal{T}(x(t))-\mathcal{T}(y(t)), \bar{c}-y(t)\rangle-\langle\mathcal{T}(y(t)), y(t)-\bar{c}\rangle\\
			&\leq&\|\mathcal{T}(x(t))-\mathcal{T}(y(t))\|\|y(t)-\bar{c}\|+h(y(t))-h(\bar{c})-\mu\|y(t)-\bar{c}\|^2\\
			&\leq&\beta\|x(t)-y(t)\|\|y(t)-\bar{c}\|+h(y(t))-h(\bar{c})-\mu\|y(t)-\bar{c}\|^2.
		\end{eqnarray*}
		This implies
		\begin{eqnarray}\label{T23}
			\lambda\langle\mathcal{T}(x(t)), \bar{c}-y(t)\rangle\leq\lambda\beta\|x(t)-y(t)\|\|y(t)-\bar{c}\|+\lambda \left[h(y(t))-h(\bar{c})\right]-\lambda\mu\|y(t)-\bar{c}\|^2.
		\end{eqnarray}
		From \eqref{2}, we have
		\begin{eqnarray}\label{T24}
			\langle x(t)-y(t), \bar{c}-y(t)\rangle-\lambda\langle\mathcal{T}(x(t)), \bar{c}-y(t)\rangle\leq\lambda\left[h(\bar{c})-h(y(t))\right].
		\end{eqnarray}
		Combining \eqref{T23} and \eqref{T24}, we obtain
		\begin{eqnarray*}
			\langle x(t)-y(t),\bar{c}-y(t)\rangle\leq\lambda\beta\|x(t)-y(t)\|\|y(t)-\bar{c}\|-\lambda\mu\|y(t)-\bar{c}\|^2.
		\end{eqnarray*}
		This implies
		\begin{eqnarray*}
			\lambda\mu\|y(t)-\bar{c}\|^2&\leq& \lambda\beta\|x(t)-y(t)\|\|y(t)-\bar{c}\|+\langle x(t)-y(t),y(t)-\bar{c}\rangle\\
			&\leq&\lambda\beta\|x(t)-y(t)\|\|y(t)-\bar{c}\|+\|x(t)-y(t)\|\|y(t)-\bar{c}\|\\
			&=&(1+\lambda\beta)\|x(t)-y(t)\|\|y(t)-\bar{c}\|.
		\end{eqnarray*}
		This further gives
		\begin{eqnarray}\label{T25}
			\|y(t)-\bar{c}\|\leq \dfrac{1+\lambda\beta}{\lambda\mu}\|x(t)-y(t)\|.
		\end{eqnarray}
		Observe that
		\begin{eqnarray}\label{T26}
			\|x(t)-\bar{c}\|\leq\|x(t)-y(t)\|+\|y(t)-\bar{c}\|.
		\end{eqnarray}
		Thus, by \eqref{T25}, inequality \eqref{T26} gives
		\begin{eqnarray*}
			\|x(t)-\bar{c}\|&\leq&\|x(t)-y(t)\|+\dfrac{1+\lambda\beta}{\lambda\mu}\|x(t)-y(t)\|\\
			&=& \dfrac{1+\lambda(\mu+\beta)}{\lambda\mu}\|x(t)-y(t)\|.
		\end{eqnarray*}
		That is 
		\begin{eqnarray}\label{T27}
			\|x(t)-y(t)\|\geq\dfrac{\lambda\mu}{1+\lambda(\mu+\beta)}\|x(t)-\bar{c}\|.
		\end{eqnarray}
		Consider the Lyapunov function $\varphi(t)=\|x(t)-\bar{c}\|^2.$\\
		Therefore, by Lemma \ref{DS1} and \eqref{T27}, we have
		\begin{eqnarray}\label{T28}
			\dot{\varphi}(t) &=& \dfrac{d}{dt}\|x(t)-\bar{c}\|^2\nonumber\\
			&=&2\langle \dot{x}(t)-\bar{c},x(t)-\bar{c}\rangle\nonumber\\
			&\leq&-2\left[1-\lambda(1+\beta^2)\right]\|x(t)-y(t)\|^2\nonumber\\
			&\leq& -2\left[1-\lambda(1+\beta^2)\right]\left(\dfrac{\lambda\mu}{1+\lambda(\mu+\beta)}\right)^2\|x(t)-\bar{c}\|^2.
		\end{eqnarray}
		Thus, by Lemma \ref{Gronwall}, we have
		\begin{eqnarray*}
			\|x(t)-\bar{x}\|^2\leq\|x(0)-\bar{c}\|^2e^{-\alpha t},
		\end{eqnarray*}
		where $\alpha=2\left[1-\lambda(1+\beta^2)\right]\left(\dfrac{\lambda\mu}{1+\lambda(\mu+\beta)}\right)^2.$ \\
		Since $x(t)$ is arbitrary, hence, by Definition \ref{defn expo} (IV),  the equilibrium point $\bar{c}$  of system \eqref{DS MVI} is globally exponentially stable for all solutions $x(t),$ $t\geq0.$
	\end{proof}
	\begin{remark}
		It is important to note that when the dynamical system \eqref{DS MVI} is explicitly discretized with respect to the time variable $t,$ then, with stepsize  $\hbar_n>0$  and a starting point $x_0\in\mathcal{H},$ it will produce the following iterative method:
		\begin{eqnarray*}
			\dfrac{x_{n+1}-x_n}{\hbar_n} +x_n = 	prox_{\lambda h}(x_n-\lambda\mathcal{T}x_n)+\lambda\mathcal{T}x_n-\lambda[\mathcal{T}(prox_{\lambda h}(x_n-\lambda\mathcal{T}x_n))] \ \ \ \forall n\geq0.
		\end{eqnarray*}
		By letting $y_n :=prox_{\lambda h}(x_n-\lambda\mathcal{T}x_n),$ then, this method can be rewritten as
		\begin{eqnarray*}
			\begin{cases}
				y_n = prox_{\lambda h}(x_n-\lambda\mathcal{T}x_n)\\ 
				\\
				x_{n+1} = (1-\hbar_n)x_n+\hbar_n[ y_n+\lambda(\mathcal{T}x_n-\mathcal{T}y_n)] \ \ \ \forall n\geq0,
			\end{cases}
		\end{eqnarray*}
		which is exactly the Tseng's forward-backward-forward method with $\{{\hbar_n}\}_{n\geq0}$ as the relaxation parameter.
		But, in a case where $\hbar_n=1,$ then, we have the following
		\begin{eqnarray*}
			\begin{cases}
				y_n = prox_{\lambda h}(x_n-\lambda\mathcal{T}x_n)\\ 
				\\
				x_{n+1} = y_n+\lambda[\mathcal{T}x_n-\mathcal{T}y_n] \ \ \ \forall n\geq0,
			\end{cases}
		\end{eqnarray*}
		which is the classical proximal forward-backward-forward method for solving MVI\eqref{MVI1} and other optimization problems proposed in \cite{Tseng}.
	\end{remark}

	\hfill
	
	\section{Numerical Experiments}\label{numerics}
\noindent We present some numerical experiments of our proposed dynamical system \eqref{DS MVI} in this section. The numerical implementation is carried out using MATLAB 2023(b), running on a personal computer equipped with an Intel(R) Core(TM) i5-10210U CPU at 2.30GHz and 8.00 GB RAM.

	\begin{example}\label{ex1}
		Let $\mathcal{C}=[3,5].$ Define $h:\mathcal{C}\longrightarrow\mathbb{R}$ and $\mathcal{T}:\mathbb{R}\longrightarrow\mathbb{R}$ by\\
		\begin{eqnarray*}
			h(\underline{u})=
			\begin{cases}
				\underline{u}^2, \ \ \mbox{if} \ \underline{u}\in\mathcal{C},\\
				+\infty, \ \ \mbox{otherwise}
			\end{cases}
		\end{eqnarray*}
		and 
		$$\mathcal{T}(\underline{u})=4-\underline{u} \  \ \ \ \forall \underline{u}\in\mathbb{R}.$$
		Where the inner product function $\langle  \cdot, \cdot\rangle$ on $\mathbb{R}$ is defined by $\langle \underline{u},\underline{v}\rangle =\underline{u} \cdot\underline{v}$ for each $\underline{u},\underline{v}\in\mathbb{R}.$ consider the MVI of finding $\bar{x}\in dom \ h$ such that
			$$\langle \mathcal{T}\underline{u}, \underline{v}-\underline{u}\rangle+h(\underline{v})-h(\underline{u})\geq0, \ \ \forall \underline{v}\in dom~h.$$ \\
			In what that follows, we show that $C_1-C_3$ of Assumption \ref{ASS1} are satisfied, thus:
			\begin{enumerate}
				\item It can easily be checked that the solution set $S^{MVI(\mathcal{T};h;\mathcal{C})}$ is nonempty. In fact, $\bar{x}=3$ is the unique solution of the MVI.
				\item It is clear that $\mathcal{T}$ is $1-$Lipschitz continuous. 
				\item Clearly, the function $h$ is proper and convex. In addition, since $|h(\underline{u})-h(\underline{v})| = |\underline{u}^2-\underline{v}^2|\leq|\underline{u}+\underline{v}||\underline{u}-\underline{v}|\leq10|\underline{u}-\underline{v}|,$ thus $h$ is Lipschitz continuous with constant $10.$ Now, we show that the $h-$pseudomonotonicity of $\mathcal{T}$ does not necessarily guarantee the pseudomonotonicity of $\mathcal{T}.$ \\
				Suppose 
				$$\langle \mathcal{T}\underline{u}, \underline{v}-\underline{u}\rangle+h(\underline{v})-h(\underline{u})\geq0, \ \ \forall \underline{v}\in [3,5].$$
				That is,
				$$(\underline{v}-\underline{u})(4+\underline{v})\geq0.$$
				But $4+\underline{v}>0,$ since $\underline{v}\in[3,5].$ This further implies that $(\underline{v}-\underline{u})\geq0.$ Hence, we have that 
				\begin{eqnarray*}
					\langle \mathcal{T}\underline{v}, \underline{v}-\underline{u}\rangle+h(\underline{v})-h(\underline{u})
					&=& (\underline{v}-\underline{u})(4+\underline{u})\\
					&\geq&0,
				\end{eqnarray*} 
				which shows that $\mathcal{T}$ is $h-$pseudomonotone on $dom~h.$ However, if we take $\underline{u} = 3$ and $\underline{v} = 5,$ then, we have that $\langle \mathcal{T}\underline{u}, \underline{v}-\underline{u}\rangle = 2>0,$ and $\langle \mathcal{T}\underline{v}, \underline{v}-\underline{u}\rangle = -2<0,$ showing that $\mathcal{T}$ is not pseudomonotone let alone monotone.
			\end{enumerate}
		\end{example}

        \hfill
            
            \begin{example}\label{ex2}
				Let $\mathcal{C}= \{x\in[-5,5]^3 ~ : ~ x_1+x_2+x_3 =0\}\subseteq\mathbb{R}^3$ and $\mathcal{T}:\mathbb{R}^3\longrightarrow\mathbb{R}^3$ be defined as
				$$\mathcal{T}(x) = \big(e^{-\|x\|^2}+q\big)Mx,$$
				where $q=0.2$ and 
				\begin{eqnarray*}
					M= \begin{bmatrix}
						1 & 0 & -1\\
						0 & 1.5 & 0\\
						-1 & 0 & 2
					\end{bmatrix}.
				\end{eqnarray*}
				Then, $\mathcal{T}$ is Lipschitz continuous with $\beta\approx 5.0679,$ and strongly pseudomonotone on $\mathbb{R}^3$ with constant $\mu:=q\cdot\lambda_{\min}\approx 0.0764,$ where $\lambda_{\min}$ is the smallest eigenvalue of $M.$ But $\mathcal{T}$ is not monotone since  for $x= (-1, 0, 0)^T,$ $y=(-2, 0,0)^T\in\mathbb{R}^3,$ we have
				$\langle \mathcal{T}(x)-\mathcal{T}(y), x-y\rangle = -0.1312<0.$\\
				
	\end{example}

\hfill
    
	\begin{example}\label{ex3}
		Consider the following logistic regression problem with an $\ell_1$-regularization term:
		$$\underset{x\in\mathbb{R}^{3}}{\min}\sum_{i=1}^{100}\log\Biggl(1+ \exp(-a_ib_i^T x)\Biggr)+\eta\|x\|_1,$$
		where $a_i \in \{-1, 1\},$ $b_i\in\mathbb{R}^{3},$ $i = 1, 2, ..., 100$ are randomly chosen, $\eta>0$  and the non-smooth $\ell_1$-regularization term is added to prevent overfitting on the given data. We choose  $\eta=2.5$, $\lambda=0.01$ and we generate $x_0$ randomly.
        
        \noindent In Figure \ref{FF3}, we plot the loss function against iterations. The loss function is the objective function being minimized during the optimization process, given as 
$$\text{Loss}(x) = \underbrace{\sum_{i=1}^{n} \log\Big(1 + \exp(-a_i b_i^\top x)\Big)}_{\text{Logistic Loss (Data Fidelity)}} + \underbrace{\eta \|x\|_1}_{\text{Regularization Term (Sparsity Encouragement)}}.$$

\noindent The Logistic Loss measures the fit of the model to the data and penalizes the misclassification of samples. On the other hand, the Regularization Term penalizes large values of $x$, encouraging sparsity in the model coefficients with $\|x\|_1$ as the $\ell_1$-norm of $x$ and $\eta > 0$ as the controls of the strength of regularization.

\subsection*{Interpretation:}

\noindent The {loss function value} at each iteration quantifies the tradeoff between model accuracy (logistic loss) and sparsity (regularization term).
A lower value of the loss function indicates a better tradeoff, with the model fitting the data while maintaining sparsity.

\subsection*{Why Plot It Against Iterations?}
\begin{itemize}
    \item \textbf{Convergence Tracking}: The loss function plot shows how the optimization algorithm reduces the objective value over iterations, indicating progress toward the optimal solution. We obtain numerically, convergence at iteration $33$ with optimal solution of $x=(-0.1899, -0.1633,   0.2324)^T$.

    \item  \textbf{Algorithm Performance}:
  The shape of the plot can help diagnose issues, such as slow convergence or improper step size.
\end{itemize}
	\end{example}

    \hfill

\begin{figure}
\begin{center}
\includegraphics[width=8cm]{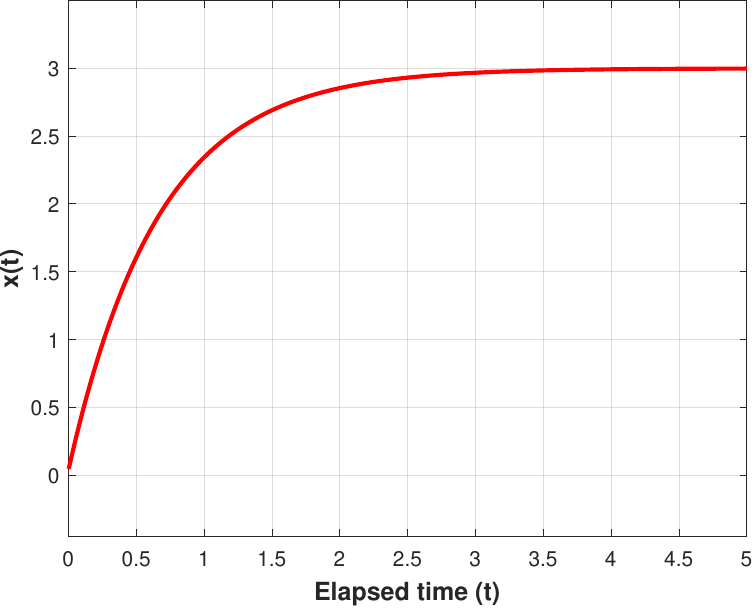}%
\includegraphics[width=8cm]{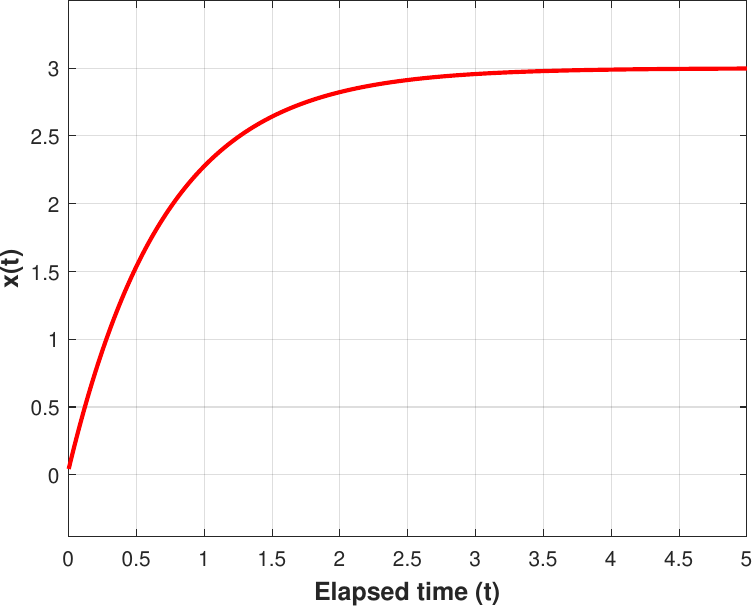}\\
\includegraphics[width=8cm]{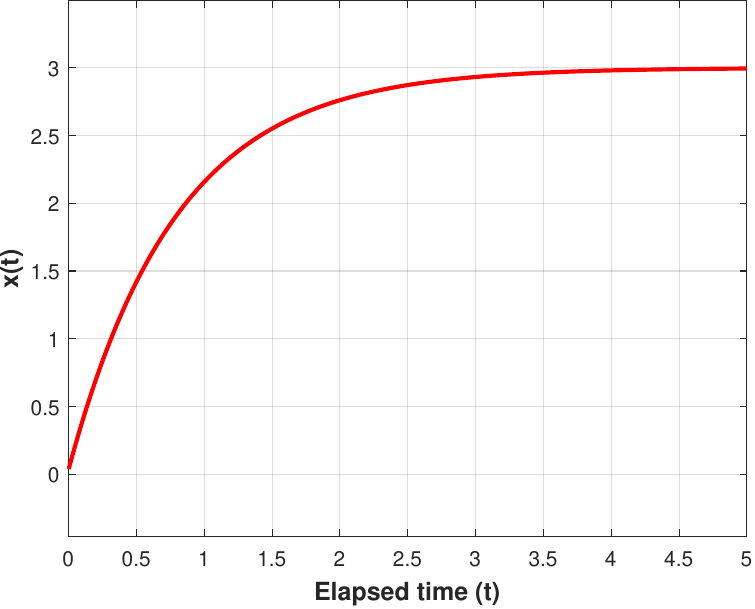}%
\end{center}
\caption{Convergence of $x(t)$ generated \eqref{DS MVI} to the unique solution $\bar{x}=3$ over time for Example \ref{ex1} with $x_0=0.1$: 
Upper Left: {$\lambda=\frac{0.99}{2}$}; Upper Right: {$\lambda=\frac{0.8}{2}$}; Bottom: {$\lambda=\frac{0.5}{2}$}.}\label{FF2}
\end{figure}

\begin{figure}
\begin{center}
\includegraphics[width=8cm]{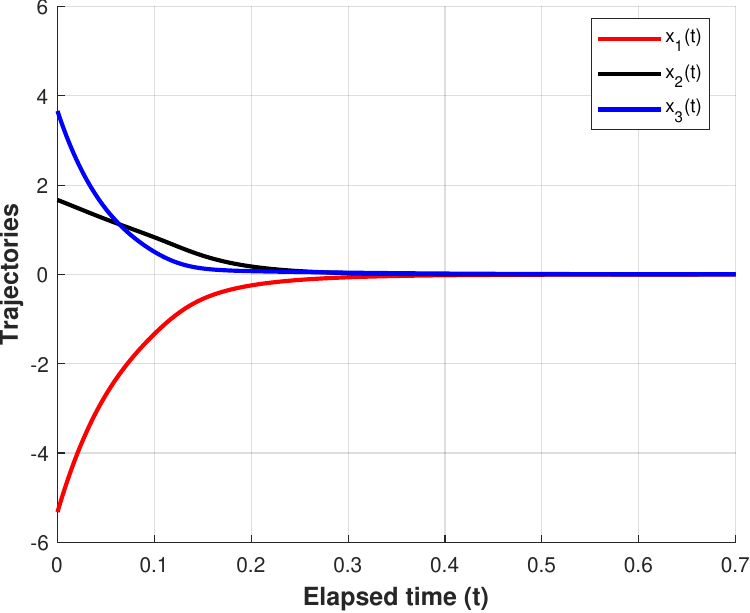}%
\includegraphics[width=8cm]{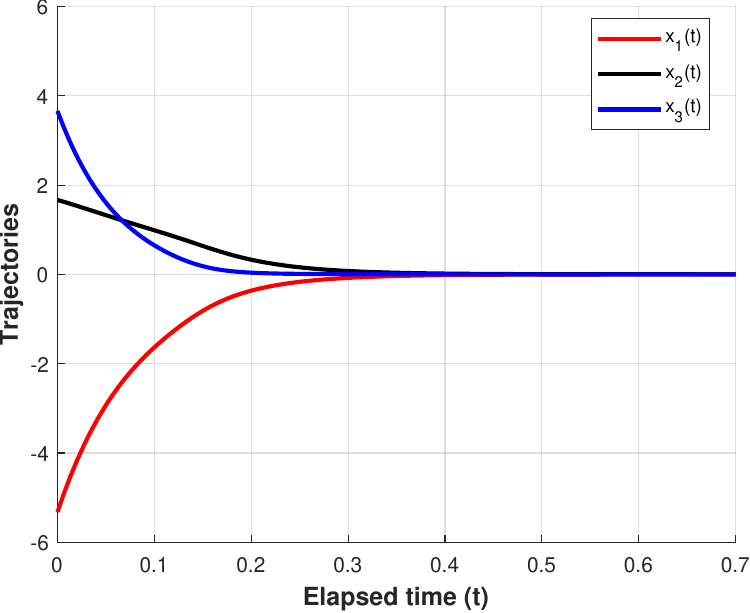}\\
\includegraphics[width=8cm]{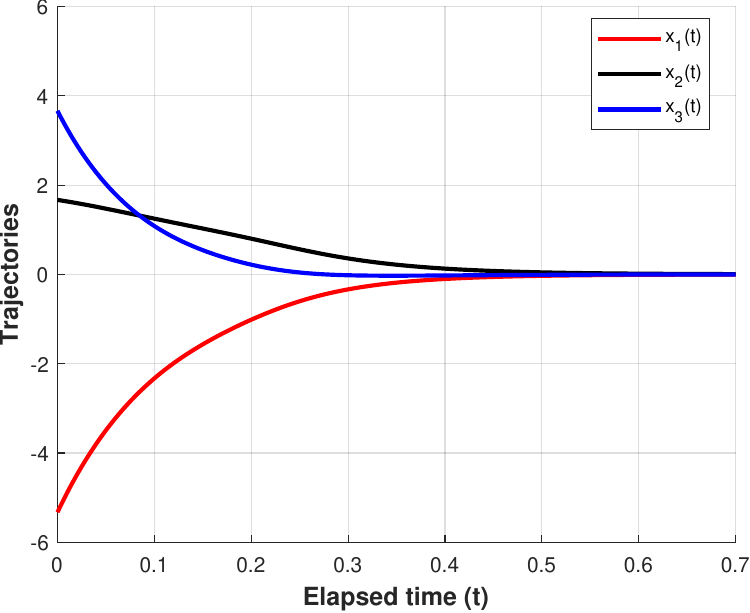}%
\end{center}
\caption{Trajectories generated by the dynamical system \eqref{DS MVI} with $x_0=(-4, 3, 5)^T$ for Example \ref{ex2}: 
Upper Left: {$\lambda=\frac{0.99}{1+\beta^2}$}; Upper Right: {$\lambda=\frac{0.8}{1+\beta^2}$}; Bottom: {$\lambda=\frac{0.5}{1+\beta^2}$}.}\label{FF2}
\end{figure}

\begin{figure}
\begin{center}
\includegraphics[width=9cm]{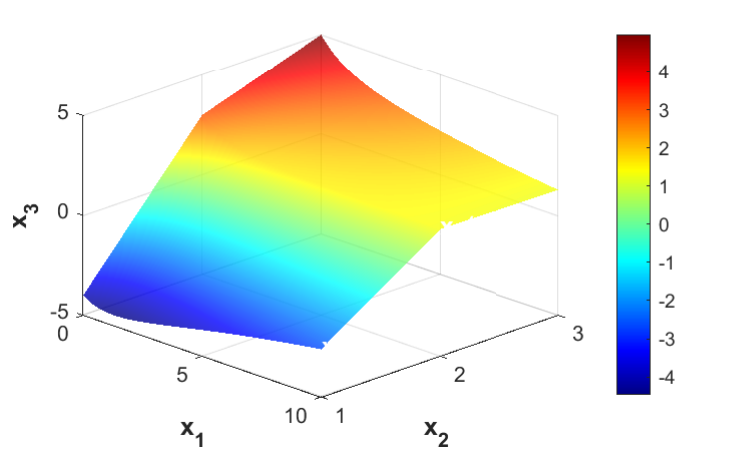}%
\includegraphics[width=8cm]{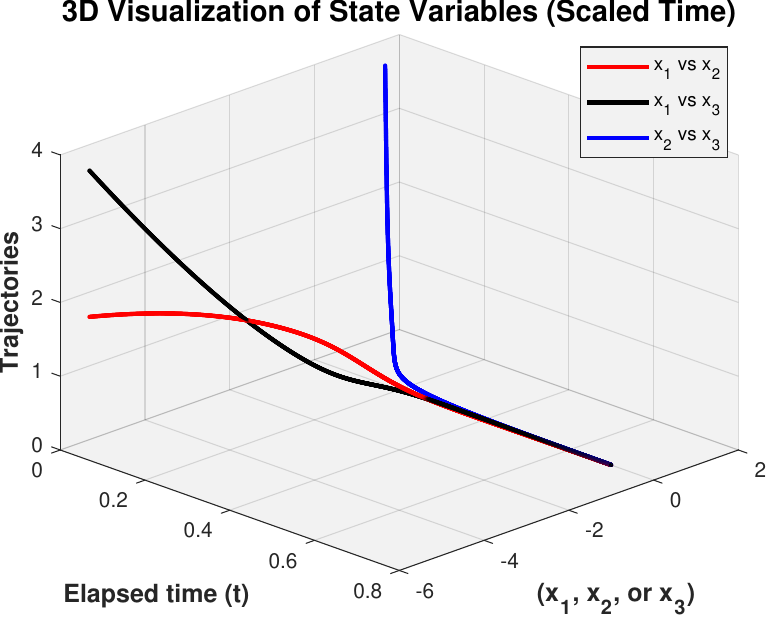}%
\end{center}
\caption{3D visualization of the trajectories generated by the dynamical system \eqref{DS MVI} for Example \ref{ex2}: $\lambda=\frac{0.99}{1+\beta^2}$.}\label{FF2}
\end{figure}

\begin{figure}
\begin{center}
\includegraphics[width=8cm]{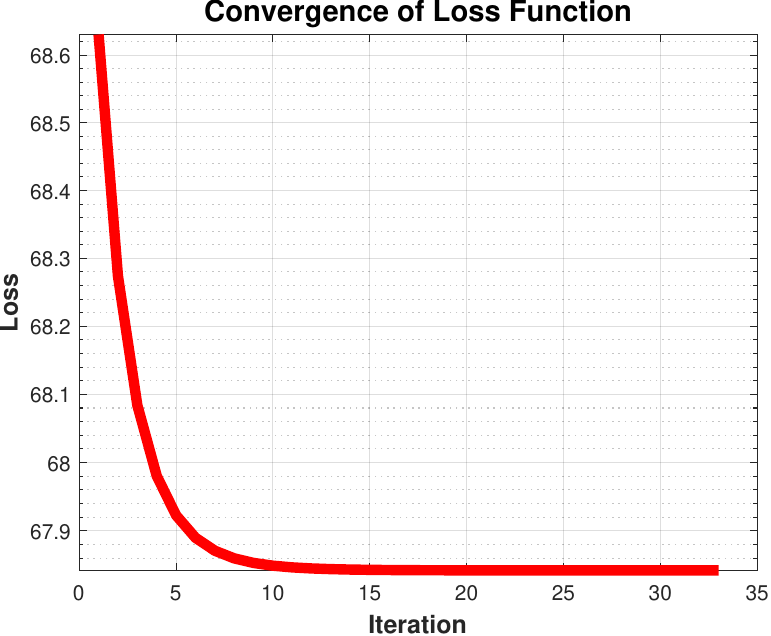}%
\includegraphics[width=8cm]{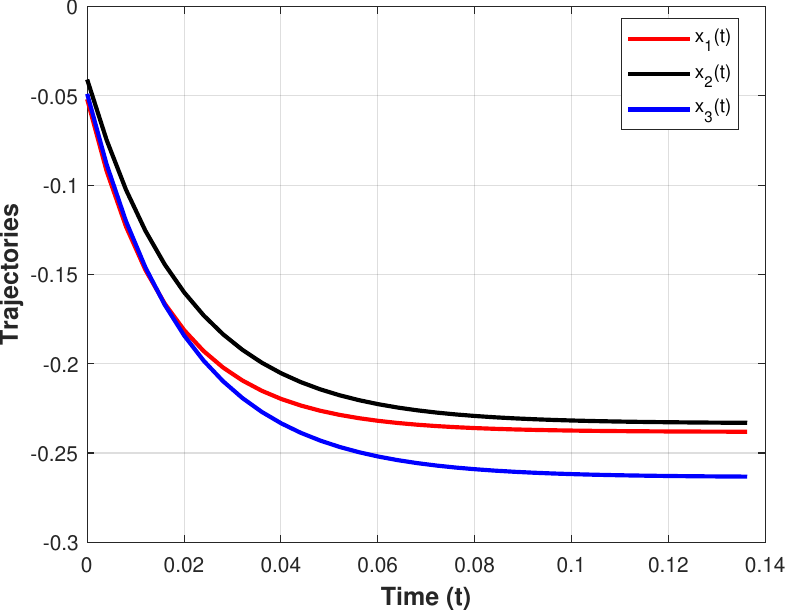}%
\end{center}
\caption{Dynamical system \eqref{DS MVI} for Example \ref{ex3}: 
Left: Shows the convergence of the loss function over iterations; Right: Convergence of the trajectories over time.}\label{FF3}
\end{figure}

\section{Conclusion}\label{Conclusion}
    \noindent In this work, we applied the Tseng's forward-backward-forward dynamical system to  solve a convex MVIP. We established that the generated trajectories of the system considered weakly convergence to a solution of the MVI when the associated operator is monotone and Lipschitz and the real-valued function is proper, convex and lower semicontinuous. In general, the Lipschitz operator and the real-valued function satisfy the general monotonicity condition. In addition, the global exponential stability of the equilibrium point of the dynamical system studied is obtained under the $h-$strongly pseudomonotonicity condition. Lyapunov-type analysis played a key role in the establishment of the convergence results. Numerical experiments are also presented to demonstrate the effectiveness of the theoretical findings. Part of our future research is to further investigate this dynamical system in a stochastic setting.
    
    \hfill

\end{document}